\documentclass[oneside,a4paper,12pt,notitlepage]{article}
\usepackage[T1]{fontenc} 
\usepackage[utf8]{inputenc} 
\usepackage{lipsum} 
\usepackage{lmodern}
\usepackage{amssymb}
\usepackage{amsthm}
\usepackage{bm}
\usepackage{mathtools}
\usepackage{braket}
\usepackage{esint}
\usepackage{todonotes}

\newcommand{\abs}[1]{{\left|#1\right|}}
\newcommand{\norma}[1]{{\left\Vert#1\right\Vert}}

\usepackage{booktabs}
\usepackage{graphicx}
\usepackage{tikz}
\usetikzlibrary{patterns}
\usepackage{multicol}
\usepackage{caption}
\usepackage{enumerate}
\usepackage[skins,theorems]{tcolorbox}
\tcbset{highlight math style={enhanced,
		colframe=black,colback=white,arc=0pt,boxrule=1pt}}
\def\XXint#1#2#3{{\setbox0=\hbox{$#1{#2#3}{\int}$}
    \vcenter{\hbox{$#2#3$}}\kern-.5\wd0}}

\theoremstyle{definition}
\newtheorem{definizione}{Definition}[section]
\theoremstyle{plain}
\newtheorem{teorema}{Theorem}[section]

\newtheorem{lemma}[teorema]{Lemma}
\newtheorem{prop}[teorema]{Proposition}

\theoremstyle{definition}
\newtheorem{esempio}{Example}[section]
\newtheorem{oss}[esempio]{Remark}

\renewcommand{\div}{\text{div}}

\DeclareMathOperator{\R}{\mathbb{R}}

\makeatletter
\newcommand{\myfootnote}[2]{\begingroup
	\def\@makefnmark{}%
	\addtocounter{footnote}{-1}%
	\footnote{\textbf{#1} #2}%
	\endgroup}
\makeatother

\numberwithin{equation}{section}

\usepackage[style = alphabetic, maxbibnames=99, maxcitenames = 99, maxbibnames = 99, maxalphanames = 99]{biblatex}
\bibliography{biblio.bib}

\usepackage{hyperref}
\hypersetup{linktoc=none, bookmarksnumbered, colorlinks=true, linkcolor=red}

\title{Isoperimetric estimates for solutions to the p-Laplacian with variable Robin boundary conditions}

\author{Vincenzo Amato, Francesco Chiacchio, Andrea Gentile}
\date{\today}

\begin{document}

\maketitle
	
\begin{abstract}
\noindent
In this paper we study the $p$-Poisson equation with Robin boundary conditions, where the Robin parameter is a function.
By means of some weighted isoperimetric inequalities, we provide various sharp bounds for the solutions to the problems under consideration. We also derive a  Faber-Krahn-type inequality. 
\\[1ex]
	
\noindent    \textsc{MSC 2020:}  35J92, 35P15, 35B45\\
    \textsc{Keywords:}  Weighted symmetrization, $p$-Laplacian, variable Robin parameter, comparison results.

\end{abstract}

\section{Introduction}
Symmetrization methods have been revealed to be a very effective and
flexible tool in the study of partial differential equations (in the sequel
just PDEs). The related bibliography  is very wide and one of the
 cornerstone of this theory is the Talenti's Theorem (see \cite{Talenti_76}), which makes use
of the Schwarz symmetrization. The technique he introduced has been refined
through the years and it has been adapted to treat a large class of linear and
nonlinear elliptic and parabolic PDEs (see, e.g., the survey papers \cite{Talenti_art_of_rearranging}, \cite{Trombetti_G_Symmetrization_methods} and the monographs \cite{Kawohl_rearrangements_and_convexity}, \cite{Kesavan_symmetrization_e_applicat} and \cite{Baernstein}). However, in most of the subsequent works, the authors, in order to
follow Talenti's original idea, consider problems whose solutions have level
sets that do not intersect the boundary of the domain where the problem is
defined.

Recently it has been discovered a method to adapt symmetrization techniques
to a class of problems that does not exhibit this feature. More precisely, in 
\cite{ANT}, the authors were able to obtain a Talenti-type comparison result  for the
Laplacian with Robin boundary conditions. For further developments in
this direction see, e.g., \cite{AGM, CGNT}, where the $p$-Laplace and Hermite operator are studied, respectively, 
while in  \cite{ACNT} and \cite{ACNT2} the authors investigate linear
problems where the Robin parameter is allowed to be a function.

Here we generalize the above results by considering the following problem

\begin{equation}
    \begin{cases}
        - \div\bigl( \abs{\nabla u}^{p-2} \nabla u \bigr) = f(x) \abs{x}^\ell & \text{in } \Omega \\[1.2ex]
        \displaystyle{\abs{\nabla u}^{p-2} \, \frac{\partial u}{\partial \nu} + \beta(x) \abs{u}^{p-2}u = 0} & \text{on } \partial \Omega ,
    \end{cases}
    \label{p_originale}
\end{equation}
where, here and throughout, $n\geq 2$,  
$\nu $ denotes the outer unit normal to $\partial \Omega $ and $\Omega $ is a bounded and Lipschitz domain of $\R^n$ such that $0 \notin \partial \Omega$.

Furthermore we will assume that
\begin{equation}
\tag{$H_1$}
    p \geq n ,  \label{p>n}
\end{equation}
\begin{equation}
\tag{$H_2$}
    -n<\ell<0,  \label{Ip_n_l}
\end{equation}
\begin{equation}
\tag{$H_3$}
   m=\inf_{\partial \Omega }\beta (x)>0\text{ \ and \ }M=\sup_{\partial \Omega
    }\beta (x)<+\infty , 
    \label{Ip_beta}
\end{equation}
and, finally,
\begin{equation}
\tag{$H_4$}
    0\leq f(x)\in L^{p^{\prime }}(\Omega ,\abs{x}^\ell \, dx),
    \label{Ip_f}
\end{equation}
where $p^{\prime }=\displaystyle\frac{p}{p-1}$.

Hypotheses  \eqref{p>n} and \eqref{Ip_n_l} ensure the validity of certain  isoperimetric inequalities, where two different weights (which are suitable powers of the distance from the origin) appear in the perimeter and in the area element, respectively (see, e.g., \cite{HS,Chiba_Horiuchi,DHHT,BBMP,CG,ABCMP} and the references therein). Indeed we need to use, in place of the more common Schwarz symmetrization, some weighted rearrangement, based on these ``double density'' isoperimetric inequalities.

A weak solution to problem (\ref{p_originale}) is a function $u\in
W^{1,p}(\Omega )$ such that 
\begin{equation}
    \int_{\Omega } \abs{\nabla u}^{p-2} \nabla u\nabla \varphi \,
    dx+\int_{\partial \Omega }\beta (x)\left\vert u\right\vert ^{p-2}u\varphi d \mathcal{H}^{n-1}(x)=\int_{\Omega }f\left\vert x\right\vert ^{\ell}\varphi \, dx
    \label{Weak_Form}
\end{equation}
for any $\varphi \in W^{1,p}(\Omega )$.

Now we need to construct the so-called symmetrized problem. That is a radial problem, of the same type of (\ref{p_originale}), whose solution will estimate the one to problem (\ref{p_originale}).

To this aim we introduce the set $\Omega ^{\sharp }$ which is the ball centered at the origin (in the sequel just centered ball) of
radius $r^{\sharp }$, \ with $r^{\sharp }$ such that 
\begin{equation*}
    \lvert \Omega ^{\sharp } \rvert _{\ell}=\int_{\Omega ^{\sharp
    }}\left\vert x\right\vert ^{\ell}dx=\int_{\Omega }\left\vert x\right\vert
    ^{\ell}dx=\left\vert \Omega \right\vert _{\ell},
\end{equation*}
and the function $f^{\sharp }(x)$ defined by the following relation 
\begin{equation*}
    \lvert \Set{ x\in \Omega :
    \abs{f(x)}>t} \rvert_{\ell} = \lvert \{ x\in \Omega :f^{\sharp}(x)>t \} \rvert _{\ell}\text{ \ for a.e. }t\geq 0.
\end{equation*}
Our symmetrized problem is the following 
\begin{equation}
\label{P_sharp}
    \begin{cases}
        - \div\bigl( \abs{\nabla v}^{p-2} \nabla v \bigr) = f^\sharp(x) \abs{x}^\ell & \text{in } \Omega^{\sharp} \\[1.2ex]
        \displaystyle{\abs{\nabla v}^{p-2} \, \frac{\partial v}{\partial \nu} + \tilde{\beta} (r^{\sharp})^{\frac{\ell}{p'}} \abs{v}^{p-2}v  = 0 } & \text{on } \partial \Omega^{\sharp}
    \end{cases}
\end{equation}
where 

\begin{equation}
    \widetilde{\beta }= \inf_{\partial \Omega} \beta(x)\abs{x}^{-\frac{\ell}{p^{\prime}}}>0,
    \label{B_tilde}
\end{equation}
since $0 \notin \partial \Omega$ and \ref{Ip_beta} holds.
 From the \eqref{B_tilde}, it immediately follows that  
\begin{equation}
    \label{lower_bon}
    \beta (x) \geq \widetilde{\beta } \abs{x}^{\frac{\ell}{p^{\prime }}}.
\end{equation}
Our main results are contained in the following theorems.

\begin{teorema}
    \label{teorema1}
    Assume that assumptions \eqref{p>n}-\eqref{Ip_f} are in force and let $u$ and $v$ be the solutions to problems (\ref{p_originale}) and (\ref{P_sharp}), respectively. Then  we have 
    \begin{equation}
        \hspace{-0.7em} \norma{u}_{L^{1}(\Omega,\abs{x}^\ell dx)}=\int_{\Omega } \abs{u(x)}  \abs{x}^\ell \, dx \leq \int_{\Omega^{\sharp}}  \abs{v(x)} \abs{x}^\ell \, dx=\norma{v}_{L^{1}(\Omega^{\sharp}, \abs{x}^\ell dx)}
        \label{L^1}
    \end{equation}
    and 
    \begin{equation}
        \hspace{-0.8em} \norma{u}_{L^{p}(\Omega, \abs{x}^\ell  dx)}^p=\int_{\Omega } \abs{u(x)}^{p}\abs{x}^{\ell}  dx\leq \int_{ \Omega^{\sharp} } \! \abs{v(x)}^p \abs{x}^\ell  dx = \norma{v}_{L^p (\Omega^{\sharp }, \abs{x}^\ell dx)}^p.
        \label{L^2}
    \end{equation}
\end{teorema}

If the source term $f(x)$ is constant
we get a stronger  result, namely a pointwise comparison between $ u^\sharp$ and $v$.
\begin{teorema}
    \label{teorema2}
	 Assume that hypotheses \eqref{p>n}-\eqref{Ip_f} are in force and suppose that $f(x) \equiv 1$ in $\Omega$. If either $p=n=2$ or if $p>2$, $n\geq 2$ and
	 \begin{equation}
	 \label{l<n+rob}
	 \displaystyle{\ell \leq -n +\frac{p-n}{p-2}}
	 \end{equation}
	 then
		\begin{equation}
		    \label{7}
		    u^\sharp (x) \leq v(x) \qquad x \in \Omega^\sharp ,
		\end{equation}
		where  $u$ and $v$ are the solutions to \eqref{p_originale} and \eqref{P_sharp} respectively.
\end{teorema}

\begin{oss}
  It is not straightforward to compare our estimates with the ones, see \cite{ANT,AGM}, obtained by means of the 
  classical Schwarz symmetrization (which corresponds to choosing $\ell=0$).
  Nevertheless, in the forthcoming  paper  \cite{ACGM}, the authors will address the asymptotic as $p \rightarrow + \infty$ of some inequalities proved in the present note. From such an analysis, it will be clear that, at least for $p$ large enough and for a special class of domains, the weighted ($\ell <0$) and the unweighted ($\ell =0$) estimates are comparable and the first ones are sharper.
  
  We finally note that there is a vast literature on Talenti's type estimates obtained via  weighted rearrangement.  Typically, the weight fulfils an isoperimetric inequality and it appears in the ellipticity of the differential operator that one wants to study. In our case, see also \cite{ACNT2}, the differential operator is the classical 
  $p$-Laplacian and the weight is somehow ``hidden'' in the boundary condition. 
This feature could represent one of the novelties of our approach.
\end{oss}

    \begin{oss}
        A straightforward computation shows that the function $v$, appearing in \eqref{7}, has the following explicit expression
        \[
        v(x) = 
        \frac{p-1}{(\ell+p)(n+\ell)^{\frac{1}{p-1}}} \bigl[ (R^{\sharp})^{\frac{\ell+p}{p-1}} - \abs{x}^{\frac{\ell+p}{p-1}}  \bigr] + \biggl( \frac{(R^{\sharp})^{\frac{\ell}{p}+1}  }{\tilde{\beta}(n+\ell)} \biggr)^{\frac{1}{p-1}}.
        \]
    \end{oss}

The first eigenvalue of the $p$-Laplace operator with the boundary conditions as in \eqref{p_originale} is the minimum of the following Rayleigh quotient

\begin{equation}
    \label{trace}
    \lambda_{1,\beta (x)} (\Omega)= \min_{\substack{\psi \in W^{1,p}(\Omega) \\ \psi\ne 0}} \frac{\displaystyle{\int_{\Omega} \abs{\nabla \psi}^p \, dx +  \int_{\partial \Omega}\beta(x) \abs{\psi}^p \, d\mathcal{H}^{n-1}(x)}}{\displaystyle{\int_{\Omega} \abs{\psi}^p\abs{x}^{\ell}  \, dx}}.
\end{equation}
Finally, Theorem \ref{teorema1} allows us to prove the following Faber-Krahn inequality (see also \cite{ANT,Bos,BD,BG,BG2,BGT,CG,FK,FNT} and the references therein).
\begin{teorema}
\label{Faber_Krahn}
    Assume hypotheses \eqref{p>n}-\eqref{Ip_beta} and let $\tilde{\beta}$ be the constant defined in \eqref{B_tilde}, then we have
    \begin{equation*}
        \lambda_{1,\beta(x)} (\Omega) \geq  \lambda_{1,\tilde{\beta}}(\Omega^\sharp).
    \end{equation*}
\end{teorema}

\vspace{0.3 cm}
The paper is organized as follows. In the next section, we recall some basic notion from the theory of rearrangements and  the weighted isoperimetric inequalities needed in the sequel. Moreover, we list some properties of the solutions to problems \eqref{p_originale} and \eqref{P_sharp}.  Section \ref{section_3} is devoted to the  proofs of Theorem \ref{teorema1} and Theorem \ref{teorema2}. 
Finally, in Section \ref{section_4}, we prove the Faber-Krahn inequality.





\section{Preliminary result}
\label{section_2}

\begin{definizione}
    Let $\Omega$ be a Lebesgue measurable subset of $\R^n$ and let $\ell \in (-n,0)$. We define the $\ell$-weighted measure of $\Omega$ as
    \begin{equation}
        \label{def_misura_l}
        \abs{\Omega}_{\ell} : = \int_{\Omega} \abs{x}^\ell \, dx,
    \end{equation}
    and the $k$-weighted perimeter of $\Omega$ as
    \begin{equation}
        \label{def_perimetro_k}
        P_{k}(\Omega) =
        \begin{cases}
            \displaystyle{\int_{\partial \Omega} \abs{x}^k \, d\mathcal{H}^{n-1}} & \text{ if } \Omega \text{ is } (n-1)\text{-rectifiable}\\[3ex]
            +\infty & \text{otherwise}.
        \end{cases}
      \end{equation}
\end{definizione}

For this family of measures, an isoperimetric inequality holds true  if $k$, $\ell$ and $n$ are suitably related.
Here we need the following result
(see either \cite{Chiba_Horiuchi} Theorem 1.3 or  \cite{ABCMP} Theorem 1.1 case \textsl{ii}).
\begin{teorema}
    \label{teo_disug_isoperimetrica}
    If assumptions  \eqref{p>n} and \eqref{Ip_n_l} hold true, then
    \begin{equation}
        \label{palla_minore_perimetro}
        P_{\frac{\ell}{p'}} (\Omega) \geq P_{\frac{\ell}{p'}} (\Omega^{\sharp}),
    \end{equation}
    where $\Omega^{\sharp}$ is the centered ball with same $\ell$-measure as $\Omega$.
\end{teorema}


\begin{oss}
    Since $\lvert \Omega^{\sharp} \rvert_{\ell}$ and $P_k(\Omega^{\sharp})$ can be explicitly computed, is it possible to rewrite \eqref{palla_minore_perimetro} in the following equivalent way
    \begin{equation}
        \label{eq_isoperimetrica}
        P_{\frac{\ell}{p'}}(\Omega) \geq \gamma_{n,\ell,p} \, \abs{\Omega}_{\ell}^{\frac{\ell(p-1)+(n-1)p}{p(n+\ell)}},
    \end{equation}
    where
    \[
    \gamma_{n,\ell,p} = (n\omega_n)^{\frac{\ell+p}{p(n+\ell)}} (\ell+n)^{\frac{\ell(p-1)+(n-1)p}{p(n+\ell)}}
    \]
    and $\omega_n$ is the Lebesgue measure of the unit ball of $\R^n$.
\end{oss}

\begin{definizione}
	Let $u: \Omega \to \R$ be a measurable function, the \emph{weighted distribution function} of $u$ is the function $\mu_{\ell} : [0,+\infty[\, \to [0, +\infty[$ defined by
	\[
	\mu_{\ell}(t)= \abs{\Set{x \in \Omega \, :\,  \abs{u(x)} > t}}_{\ell}.
	\]
\end{definizione}
In the following, we will omit the $\ell$ and just write $\mu$.

\begin{definizione}
	Let $u: \Omega \to \R$ be a measurable function, the \emph{weighted decreasing rearrangement} of $u$, denoted by $u^\ast$, is the distribution function of $\mu $. 
	
    The \emph{weighted rearrangement} of $u$ is the function $u^\sharp $ whose level sets are centered balls with the same $\ell$-measure as the level sets of $\abs{u}$. More precisely,
    \[
    u^\sharp (x)= u^* \bigl( \, \bigl \lvert B_{\abs{x}} \bigr \rvert_{\ell} \, \bigr) = u^* \biggl( \frac{n \omega_n}{\ell+n} \, \abs{x}^{\ell+n} \biggr).
    \]
    with $B_{\abs{x}}$ the centered ball of radius $\abs{x}$.
\end{definizione}

\begin{definizione}
    If $p \in [1,+\infty)$ we will denote by $L^p(\Omega, \abs{x}^\ell \, dx)$ the space of all measurable functions such that
    \[
    \norma{u}_{L^p(\Omega, \abs{x}^\ell \, dx)} := \biggl( \int_{\Omega} \abs{u}^p \abs{x}^\ell \, dx \biggr)^{\frac{1}{p}} < + \infty.
    \]
\end{definizione}

It is easily checked that $u$, $u^*$ e $u^\sharp$ are equi-distributed, hence
\begin{equation}
\label{equi}
\displaystyle{\norma{u}_{L^p(\Omega, \abs{x}^\ell \, dx)}=\norma{u^*}_{L^p(0, \abs{\Omega}_{\ell})}=\lVert{u^\sharp}\rVert_{L^p(\Omega^\sharp, \abs{x}^\ell \, dx)}} \qquad \forall \, 1 \leq p <+\infty.
\end{equation}

We also recall the definition of weighted Lorentz spaces.
\begin{definizione}
    Let $0<p<+\infty$ and $0<q\leq +\infty$. The weighted Lorentz space $L^{p,q}(\Omega, \abs{x}^\ell \, dx)$ is the space of those functions $f(x)$ such that the quantity
    \begin{equation}
    \label{deff_norma_lorentz}
    \norma{g}_{L^{p,q}(\Omega, \abs{x}^\ell \, dx)} :=
        \begin{cases}
            \displaystyle{ p^{\frac{1}{q}} \biggl( \int_{0}^{\infty}  t^q \mu(t)^{\frac{q}{p}}\, \frac{dt}{t} \biggr)^{\frac{1}{q}}} & 0<q<\infty\\[2ex] \displaystyle{\sup_{t>0} \, (t^p \mu(t))} & q=\infty
    	\end{cases}
    \end{equation}
    is finite.
\end{definizione}
Let us observe that for $p=q$ the Lorentz space coincides with the $L^p(\Omega, \abs{x}^\ell \, dx)$ space, as a consequence of the well-known \emph{Cavalieri's Principle}
\[
\int_\Omega \abs{g}^p \abs{x}^\ell \, dx = p \int_0^{+\infty} t^{p-1} \mu(t) \, dt.
\]
By  the definition of decreasing rearrangement, the following result holds true  (see for instance \cite{Chong_Rice}, \cite{Kawohl_rearrangements_and_convexity} and \cite{Kesavan_symmetrization_e_applicat}).
\begin{prop}
    Let $u \in L^1(\Omega, \abs{x}^\ell \, dx)$ be a non negative function and let $E \subset \Omega$ be a measurable set, then it holds
    \begin{equation}
    \label{E}
        \int_{E} u(x) \abs{x}^\ell \, dx \leq \int_0^{\abs{E}_{\ell}} u^*(s) \, ds.
    \end{equation}
\end{prop}

Now we turn our attention on  problems \eqref{p_originale} and \eqref{P_sharp}.
\begin{prop}
    \label{proposizione_2}
    If the assumptions \eqref{Ip_n_l} and \eqref{Ip_beta} are fulfilled  and if $\Omega$ is a bounded and Lipschitz domain of $\R^n$ then $W^{1,p}(\Omega)$ is compactly embedded in $L^p(\Omega, \abs{x}^\ell \, dx)$.
\end{prop}

\begin{prop}
    \label{proposizione_3}
     Problems \eqref{p_originale} and \eqref{P_sharp} admit a unique solution.
\end{prop}

The proofs of the above propositions  will be postponed in the Appendix, since they require rather standard arguments.





Let us observe that $u \geq 0$ in $\Omega$. Indeed, choosing $\psi = u^- = \max \Set{- u, 0}$ as test function in \eqref{Weak_Form}, we have
\[
0 \geq - \int_{\Omega} \lvert \nabla u^- \rvert^p \, dx - \int_{\partial \Omega} \beta(x) (u^-)^p \, d\mathcal{H}^{n-1}(x) = \int_{\Omega} (u^-) f \abs{x}^\ell \, dx,
\]
thus $u^- = 0$ a.e. in $\Omega$.

Furthermore we have 
\begin{equation}
		\label{minima_eq}
		u_m = \inf_{\Omega} u \leq  \min_{\overline{\Omega^\sharp}} v= v_m.
	\end{equation}
Indeed, using the fact that $v(x) = v_m$ $\forall x \in \partial \Omega^{\sharp}$ , taking into account of \eqref{equi} and choosing $\psi \equiv 1$ in the weak formulation of \eqref{p_originale} and \eqref{P_sharp}, respectively, we get
\begin{align*}
    v_m^{p-1}  P_{\frac{\ell}{p'}}(\Omega^\sharp) &= \int_{\partial \Omega^\sharp}\abs{x}^{\frac{\ell}{p'}} v(x)^{p-1} \, d\mathcal{H}^{n-1}(x)= \frac{1}{\tilde{\beta}}\int_{\Omega^\sharp} f^\sharp \abs{x}^\ell\, dx \\
    &=\frac{1}{\tilde{\beta}} \int_{\Omega} f \abs{x}^\ell \, dx=\frac{1}{\tilde{\beta}}\int_{\partial \Omega} \beta(x) u(x)^{p-1} \, d\mathcal{H}^{n-1}(x).
\end{align*}
In turn, by \eqref{lower_bon} and using the isoperimetric inequality  \eqref{palla_minore_perimetro}, we get
\begin{align*}
    \frac{1}{\tilde{\beta}}\int_{\partial \Omega} \beta(x) u(x)^{p-1} \, d\mathcal{H}^{n-1}(x) & \geq  u_m^{p-1}  \int_{\partial \Omega}  \abs{x}^{\frac{\ell}{p'}} \, d\mathcal{H}^{n-1}(x) \\
    & = u_m^{p-1} P_{\frac{\ell}{p'}}(\Omega) \geq  u_m^{p-1} P_{\frac{\ell}{p'}}(\Omega^\sharp).
\end{align*}
Gathering the inequalities above, we conclude that
\[
    v_m^{p-1} P_{\frac{\ell}{p'}}(\Omega^\sharp) \geq u_m^{p-1} P_{\frac{\ell}{p'}}(\Omega^\sharp)
\]
and  the claim follows.

Furthermore inequality \eqref{minima_eq} implies that 
\begin{equation}
	\label{mf}
	\mu (t) \leq \phi (t) = \abs{\Omega} _{\ell}\quad \forall t \leq v_m.
\end{equation}




\subsection{Some useful Lemmata}
Let $u$ be the solution to \eqref{p_originale}. For $t\geq 0$, we define
\[
    U_t=\left\lbrace x\in \Omega : u(x)>t\right\rbrace \quad \partial U_t^{int}=\partial U_t \cap \Omega, \quad \partial U_t^{ext}=\partial U_t \cap \partial\Omega,
\]
and
\[
    \mu(t)=\abs{U_t}_{\ell} \quad P_u(t)=P_{\frac{\ell}{p'}}(U_t).
\]

Analogously, if $v$ is the solution to \eqref{P_sharp}, we set
\[
    V_t=\{ x\in \Omega^\sharp : v(x)> t \}, \quad \phi(t)=\abs{V_t}_{\ell}, \quad P_v(t)=P_{\frac{\ell}{p'}}(V_t).
\]

As it is easy to check $v(x) \equiv v^\sharp(x)$, 
therefore for $0\leq t\leq v_m$, $V_t=\Omega^\sharp$, while, for
$t >v_m$
$V_t$ is a centered ball contained in $\Omega^{\sharp}$.


\begin{lemma}[Gronwall]
	\label{lemma_Gronwall}
	Let $\xi(t): [\tau_0 , + \infty[ \,\to \R$ be a continuous and differentiable function satisfying, for some non negative constant $C$, the following differential inequality
	\begin{equation}
	    \label{eq_Gronwall_hp}
	    \tau \xi' (\tau ) \leq (p-1) \xi(\tau) + C \quad \forall \tau \geq \tau_0 >0.
	\end{equation}
	Then we have
	\begin{itemize}
		\item[(i)] $\displaystyle{ \xi(\tau) \leq \left(\xi(\tau_0) + \frac{C}{p-1}\right) \left( \frac{\tau}{\tau_0}\right)^{p-1} - \frac{C}{p-1}   \quad \forall \tau \geq \tau_0}$;
		\vspace{0.7em}
		
		\item[(ii)]$\displaystyle{ \xi'(\tau) \leq \left( \frac{(p-1)\xi(\tau_0 )+ C}{\tau_0}\right) \left( \frac{\tau}{\tau_0}\right)^{p-2} \quad \forall \tau \geq \tau_0}$.
	\end{itemize}
\end{lemma}
\begin{proof}
    Dividing both sides of the differential inequality \eqref{eq_Gronwall_hp} by $\tau^p$,  we get
	\begin{equation*}
	     \frac{\xi'(\tau)}{\tau^{p-1}}-(p-1) \frac{\xi(\tau)}{\tau^p}  = \left(\frac{\xi(\tau)}{\tau^{p-1}}\right)' \leq \frac{C}{\tau^p}.
	\end{equation*}
	Integrating the  last inequality on $(\tau_0 , \tau )$ we obtain
	\[
        \int_{\tau_0}^\tau \left(\frac{\xi(\tau)}{\tau^{p-1}}\right)' \, d\tau \leq  \int_{\tau_0}^\tau \frac{C}{\tau^p} \, d\tau,
    \]
  and  therefore
    \[
        \xi(\tau) \leq \left(\xi(\tau_0) + \frac{C}{p-1}\right) \left( \frac{\tau}{\tau_0}\right)^{p-1} - \frac{C}{p-1},
	\]
	which gives \emph{(i)}.
	
	The second statement of the Lemma, claim \emph{(ii)}, is a direct consequence of the first one. 
\end{proof}

\begin{lemma}
    Let $u$ and $v$ be the solutions to \eqref{p_originale} and \eqref{P_sharp} respectively. Then, for almost every $t >0$, it holds
	\begin{equation} 
	    \label{3.2}
	    \begin{split}
	        \hspace{-0.65em} &\gamma_{n,\ell,p} \, \mu(t)^{\left(\frac{\ell(p-1) +p(n-1)}{p(\ell+n)}\right)\frac{p}{p-1}} \leq  \\
	        &  \left(\int_0^{\mu (t)}f^\ast (s ) \, ds\right)^{\frac{1}{p-1}}
	        \left( - \mu'(t) + \frac{1}{\tilde{\beta} ^{\frac{1}{p-1}}}\int_{\partial U_t^\text{ext}} \frac{\abs{x}^{\frac{\ell}{p'}}}{u} \, d\mathcal{H}^{n-1}(x)\right),
	    \end{split}
	\end{equation}
	where 
	\[
	    \gamma_{n,\ell,p} = (n \omega_n)^{\frac{\ell+p}{p(n+\ell)}} (\ell+n)^{\frac{(p-1)\ell+(n-1)p}{p(n+\ell)}}
	\]
	and 
	\begin{equation} 
	    \label{3.1}
	    \begin{split}
	        \hspace{-0.65em} & \gamma_{n,\ell,p} \, \phi(t)^{\left(\frac{\ell(p-1) +p(n-1)}{p(\ell+n)}\right)\frac{p}{p-1}} = \\
	        & \left(\int_0^{\phi (t)}f^\ast (s ) \, ds\right)^{\frac{1}{p-1}} \left( - \phi'(t) + \frac{1}{\tilde{\beta} ^{\frac{1}{p-1}}}\int_{\partial V_t^\text{ext}} \frac{\abs{x}^{\frac{\ell}{p'}}}{v} \, d\mathcal{H}^{n-1}(x)\right).
	    \end{split}
	\end{equation}
\end{lemma}

\begin{proof} Let $t >0$ e $h >0$, we use the following test function in \eqref{Weak_Form}
	\begin{equation*}
    	\varphi (x)= 
    	\begin{cases}
        	0 & \text{ if }u < t \\
        	u-t & \text{ if }t< u < t+h \\
        	h & \text{ if }u > t+h.
    	\end{cases}
	\end{equation*}
	We get
	\begin{equation*}
    	\begin{split}
        	\int_{U_t \setminus U_{t+h}} \abs{\nabla u}^p\, dx & +  h \int_{\partial U_{t+h}^{ ext }} \beta(x) u^{p-1}  \, d\mathcal{H}^{n-1}(x)  \\
        	& + \int_{\partial U_{t}^{ ext }\setminus \partial U_{t+h}^{ ext }} \beta(x)  u^{p-1} (u-t) \, d\mathcal{H}^{n-1}(x) \\
        	& =  \int_{U_t \setminus U_{t+h}}  f \abs{x}^\ell(u-t ) \, dx + h \int_{U_{t+h}} f \abs{x}^\ell\, dx.
    	\end{split}
	\end{equation*}
	Dividing by $h$, using coarea formula and letting $h$ go to 0, we have for a.e. $t>0$
	\begin{equation*}
	    \int_{\partial U_t} g(x) \, d\mathcal{H}^{n-1}(x) = \int_{U_{t}} f \abs{x}^\ell\, dx ,
	\end{equation*}
	with
	\begin{equation}
	    \label{cosag}
    	g(x)= 
    	\begin{cases}
    	    \abs{\nabla u }^{p-1} & \text{ if }x \in \partial U_t^{ int },\\[1em]
    	    \beta(x) u ^{p-1}& \text{ if }x \in \partial U_t^{ ext }.
        \end{cases}
	\end{equation}
	Using  isoperimetric inequality \eqref{eq_isoperimetrica} and  H\"older inequality, for a.e. $t>0$, we obtain
	\begin{equation}
	\label{split}
        \begin{split}
            & \gamma_{n,\ell,p} \mu(t)^{\left(\frac{\ell(p-1) +p(n-1)}{p(\ell+n)}\right)} \\
            & \leq P_u(t) = \int_{\partial U_t} \abs{x}^{\frac{\ell}{p'}}\,  d\mathcal{H}^{n-1}(x) \\
            & \leq \left(\int_{\partial U_t}g\, d\mathcal{H}^{n-1}(x)\right)^{\frac{1}{p}} \left(\int_{\partial U_t}\frac{\abs{x}^\ell}{g^{\frac{1}{p-1}}} \, d\mathcal{H}^{n-1}(x)\right)^{1-\frac{1}{p}}.
        \end{split}
	\end{equation}
	Since, for a.e. $t >0 $ it holds that
	\[
	-\mu'(t) = \int_{\partial U_t^{int}} \frac{\abs{x}^\ell}{\abs{\nabla u}} \, d \mathcal{H}^{n-1},
	\]
     by \eqref{lower_bon}, \eqref{E}, \eqref{cosag} and \eqref{split}, we get
	\begin{align*}
	    & \gamma_{n,\ell,p} \mu(t)^{\left(\frac{\ell(p-1) +p(n-1)}{p(\ell+n)}\right)} \\
	    &\leq \left(\int_{\partial U_t} \! \! \! g\,   d\mathcal{H}^{n-1}(x)\right)^{\frac{1}{p}} \! \left( \int_{\partial U_t^{ int }}\frac{\abs{x}^\ell}{\abs{\nabla u}}\, d\mathcal{H}^{n-1}(x) +\frac{1}{\tilde{\beta}^{\frac{1}{p-1}}} \int_{\partial U_t^{ ext }} \! \!  \frac{\abs{x}^{\frac{\ell}{p'}}}{u} \,  d\mathcal{H}^{n-1}(x) \right)^{\frac{1}{p'}} \\
	    &\leq \left(\int_0^{\mu(t)} f^\ast (s) \, ds\right)^{\frac{1}{p}} \left( -\mu'(t) +\frac{1}{\tilde{\beta}^{\frac{1}{p-1}}} \int_{\partial U_t^{ ext }}\frac{\abs{x}^{\frac{\ell}{p'}}}{u} \,  d\mathcal{H}^{n-1}(x) \right)^{\frac{1}{p'}}
	\end{align*}
	for a.e. $t \in [0, \sup_{\Omega} u )$.
	
	Then \eqref{3.2} follows. Replacing $u$ with $v$, the solution to \eqref{P_sharp}, all the inequalities are verified as equalities, so we obtain \eqref{3.1}.
\end{proof}
\begin{lemma}
	\label{lemma3.3}
	Let $u$ and $v$ be the solutions to \eqref{p_originale} and \eqref{P_sharp} respectively. Then for all $\tau \geq v_m $ we have
	\begin{equation}
    	\label{3.11}
    	\int_0^\tau t^{p-1} \left(\int_{\partial U_t^{ ext } } \frac{\abs{x}^{\frac{\ell}{p'}}}{ u(x) } \, d \mathcal{H}^{n-1}(x)\right) \, dt \leq  \frac{1}{p\tilde{\beta}} \int_0^{\abs{\Omega}_{\ell}} f^\ast(s) \,ds
	\end{equation}
	and
	\begin{equation}
	    \label{3.10}
    	\int_0^\tau t^{p-1} \left(\int_{\partial V_t \cap \partial\Omega^\sharp } \frac{\abs{x}^{\frac{\ell}{p'}}}{ v(x) } \, d \mathcal{H}^{n-1}(x)\right) \, dt =  \frac{1}{p\tilde{\beta}} 
    	\int_0^{\lvert \Omega^{\sharp} \rvert_{\ell}} f^\ast(s) \,ds.
	\end{equation}
\end{lemma}
\begin{proof}
    Firstly we show \eqref{3.11}. Fubini's Theorem ensures that

	\begin{equation}
	\label{parteA}
    	\begin{split}
        	& \int_0^\infty \! \! \tau^{p-1} \left(\int_{\partial U_\tau^{ ext } } \frac{\abs{x}^{\frac{\ell}{p'}}}{ u(x) } d \mathcal{H}^{n-1}(x)\right)  d\tau \\
        	& =\int_{\partial \Omega} \left(\int_0^{u(x)} \frac{\tau^{p-1}}{u(x)} \, d\tau \right) \abs{x}^{\frac{\ell}{p'}}\, d\mathcal{H}^{n-1}(x)\\
        	&=\frac{1}{p}\int_{\partial \Omega} u(x)^{p-1}  \abs{x}^{\frac{\ell}{p'}}\, d\mathcal{H}^{n-1}(x).
    	\end{split}
	\end{equation}
	Choosing $\varphi \equiv 1$ as test function in \eqref{Weak_Form} and using \eqref{lower_bon}, we obtain
\begin{equation}
\label{parteB}
	\int_{\partial \Omega} \tilde{\beta} \abs{x}^{\frac{\ell}{p'}} u^{p-1} \, d\mathcal{H}^{n-1} \leq \int_{\partial \Omega} \beta(x) u^{p-1} \, d\mathcal{H}^{n-1} = \int_{\Omega} f \abs{x}^\ell \, dx.  
\end{equation}
From \eqref{parteA} and \eqref{parteB} we immediately get	
	\[
	\int_0^\infty \tau^{p-1} \left(\int_{\partial U_\tau^{ ext } } \frac{\abs{x}^{\frac{\ell}{p'}}}{ u(x) } \, d \mathcal{H}^{n-1}(x)\right) \, d\tau \leq \frac{1}{p \tilde{\beta}} \int_0^{\lvert \Omega \rvert_{\ell}}f^\ast (s) \, ds,
	\]
	and, arguing in the same way, we deduce
	\[
    \int_0^\infty \tau^{p-1} \left(\int_{\partial V_\tau \cap \partial\Omega^{\sharp} } \frac{\abs{x}^{\frac{\ell}{p'}}}{ v(x) } \, d \mathcal{H}^{n-1}(x)\right) \, d\tau = \frac{1}{p \tilde{\beta}} \int_0^{\lvert \Omega \rvert_{\ell}}f^\ast (s) \, ds.
	\]
    Hence for any $ t\ge0$, we have
	\[
	\int_0^t \tau^{p-1} \left(\int_{\partial U_\tau^{ ext } } \frac{\abs{x}^{\frac{\ell}{p'}}}{ u(x) } \, d \mathcal{H}^{n-1}(x)\right) \, d\tau \leq\frac{1}{p \tilde{\beta}} \int_0^{\lvert \Omega \rvert_{\ell}}f^\ast (s) \, ds.
	\]
    Since $\partial V_t \cap \partial \Omega^\sharp$ is empty for any $t\geq v_m$, we get
	\[
	\int_0^t \tau^{p-1} \left(\int_{\partial V_\tau \cap \partial\Omega^{\sharp} } \frac{\abs{x}^{\frac{\ell}{p'}}}{ v(x) } \, d \mathcal{H}^{n-1}(x)\right) \, d\tau = \frac{1}{p \tilde{\beta}} \int_0^{\lvert \Omega \rvert_{\ell}}f^\ast (s) \, ds.
	\]  
	 Lemma \ref{lemma3.3} is hence proved.
\end{proof}

\section{Main results}
\label{section_3}

In this Section we prove Theorem \ref{teorema1} and Theorem \ref{teorema2}.
\begin{proof}[Proof of Theorem \ref{teorema1}]
	Let 
	\[
	\displaystyle{\delta_1 = 1 -{{\left(\frac{\ell(p-1) +p(n-1)}{(p-1)(\ell+n)}\right).}}}
	\]
	Note that, since $p\geq n$, $\delta_1$ is	a positive constant.
	
	Multiplying both sides of \eqref{3.2} by $t^{p-1} \mu (t)^{\delta_1}$ and integrating over $(0,\tau)$ with $\tau \geq v_m$, by Lemma \ref{lemma3.3}, we have
	\begin{equation}
		\label{3.12}
		\begin{split}
		\gamma_{n,\ell,p}	\int_0^\tau  t^{p-1}\mu(t) \, dt &
			\leq  	
			\int_0^\tau \left(- \mu'(t) \right) t^{p-1}\mu(t)^{\delta_1} \left(\int_0^{\mu (t)}f^\ast (s ) \, ds\right)^{\frac{1}{p-1}} \, dt \\ 
			&+ \frac{\abs{\Omega}_{\ell}^{\delta_1}}{p\tilde{\beta} ^{\frac{p}{p-1}}} \left(\int_0^{\abs{\Omega}_{\ell}}f^\ast (s ) \, ds\right)^{\frac{p}{p-1}}.	
		\end{split} 
	\end{equation}
	Setting
	\[
	\displaystyle{F(\ell)= \int_0^\ell \omega^{\delta_1} 
	\left(\int_0^\omega f^\ast (s) \, ds \right)^{\frac{1}{p-1}} \, d\omega},
	\]
	and, integrating by parts both sides of the last inequality, we obtain
    \begin{equation*}
    	\begin{split}
    		& \tau^{p-1} \left(
    		\gamma_{n,\ell,p}	\int_0^\tau  \mu(t) \, dt+ F(\mu(\tau)) \right) \\
    		& \leq  (p-1)  \int_0^\tau t^{p-2} 
    		\left(
    	    \gamma_{n,\ell,p}	\int_0^t \mu(s)\, ds+ F(\mu(t))\right)\, dt \\ 
    		& +\frac{\abs{\Omega}_{\ell}^{\delta_1}}{p\tilde{\beta} ^{\frac{p}{p-1}}} \left(\int_0^{\abs{\Omega}_{\ell}}f^\ast (s ) \, ds\right)^{\frac{p}{p-1}}.
    	\end{split} 
    \end{equation*}
    We observe that all the hypotheses of Gronwall's Lemma \ref{lemma_Gronwall} are satisfied with 
    \[
    \xi(\tau)=\xi_1(\tau)= \int_0^\tau t^{p-2}
    \left( \gamma_{n,\ell,p} \int_0^t  \mu(s)\, ds+ F(\mu(t))\right)\, dt
    \]
    and
    \[
    C= C_1 =
    \frac{
    \abs{\Omega}_{\ell}^{\delta_1}}{p\tilde{\beta} ^{\frac{p}{p-1}}} \left(\int_0^{\abs{\Omega}_{\ell}}f^\ast (s ) \, ds\right)^{\frac{p}{p-1}}.
    \]
    Applying such a Lemma  with $\tau_0= v_m$, we infer that 
    for any $\tau \geq v_m$ it holds
    \begin{equation*}
    		\tau^{p-2} 
    		\left(
    \gamma_{n,\ell,p}		\int_0^\tau  \mu(s)\, ds+ F(\mu(\tau))\right)
    \leq 
    \left( \frac{(p-1)\xi_1(v_m )+ C_1}{v_m}\right) \left( \frac{\tau}{v_m}\right)^{p-2} .
    \end{equation*}
  
    Replacing $\mu(t)$ with $\phi(t)$ the previous inequality holds as an equality. Therefore for any  $\tau \geq v_m$ it holds that 
    \begin{equation}
    \label{>t_m}
    	\gamma_{n,\ell,p}	 \int_0^\tau  \mu(s)\, ds+ F(\mu(\tau))
    		 \leq 
    	\gamma_{n,\ell,p}	 \int_0^\tau  \phi(s)\, ds+ F(\phi(\tau)).
    \end{equation}
    On the other hand,  since $\mu(t) \leq \phi(t)=\abs{\Omega}_{\ell}, \, \forall t \leq v_m$, and $F(\ell)$ is an increasing function, we have
    that for any  $\tau \leq v_m$ it holds that
     \begin{equation}
    \label{<t_m}
         \gamma_{n,\ell,p} \int_0^\tau  \mu(s)\, ds+ F(\mu(t))
         \leq
        \gamma_{n,\ell,p}   \int_0^\tau   \phi(s)\, ds+ F(\phi(t))
    \end{equation}
    Combining inequalities \eqref{>t_m} and \eqref{<t_m} we conclude that 
    for any $\tau \geq 0$ it holds
    \[
   \gamma_{n,\ell,p} \int_0^\tau  \mu(s)\, ds+ F(\mu(\tau))
   \leq 
  \gamma_{n,\ell,p} \int_0^\tau  \phi(s)\, ds+ F(\phi(\tau)).
    \]
Since
\begin{equation*}
\lim_{\tau \rightarrow +\infty }F\left( \mu (\tau )\right) =\lim_{\tau
\rightarrow +\infty }F\left( \phi (\tau )\right) =0,
\end{equation*}
letting $\tau$ goes to $+\infty$, we conclude that
    \[
    \int_0^\infty \mu(t) \, dt \leq \int_0^\infty \phi(t) \, dt ,
    \]
    and hence
    \[
    \norma{u}_{L^{1}(\Omega,\abs{x}^\ell \, dx)} \leq \norma{v}_{L^{1}(\Omega^\sharp,\abs{x}^\ell \, dx)} .
    \]
  
    \vspace{1 cm}
    
\noindent  Now   let us show \eqref{L^2}. Firstly let us observe that it is enough to verify that
    \begin{equation}
    	\label{basta}
    	\int_0^\infty t^{p-1} \mu(t) \, dt \leq \int_0^\infty  t^{p-1} \phi(t) \, dt.
    \end{equation}
    Integrating by parts the first term on the right-hand side in \eqref{3.12}, and then letting $\tau$ go to $+ \infty$, we obtain
    \[
    \gamma_{n,\ell,p}  \int_0^\infty  t^{p-1}\mu(t) \, dt
    \leq
    (p-1)\int_0^\infty t^{p-2} F(\mu(t)) \, dt
    +
    \frac
    { \abs{\Omega}_{\ell}^{\delta_1 }}
    {p\tilde{\beta} ^{\frac{p}{p-1}}} 
    \left(\int_0^{\abs{\Omega}_{\ell}}f^\ast (s ) \, ds\right)^{\frac{p}{p-1}},
    \]
    and similarly we have
    \[
    \gamma_{n,\ell,p}   \int_0^\infty  t^{p-1}\phi(t) \, dt =\  	(p-1)\int_0^\infty t^{p-2} F(\phi(t)) \, dt 
    +        
    \frac  { \abs{\Omega}_{\ell}^{\delta_1 }}{p\tilde{\beta} ^{\frac{p}{p-1}}} \left(\int_0^{\abs{\Omega}_{\ell}}f^\ast (s ) \, ds\right)^{\frac{p}{p-1}}.
    \]
    Thus if we prove
    \begin{equation}
        \label{scorciatoia}
        \int_0^\infty t^{p-2} F(\mu(t)) \, dt \leq \int_0^\infty t^{p-2} F(\phi(t)) \, dt
    \end{equation}
    we get the claim \eqref{basta}.
   
    Let 
    \[
    \theta_1 
    = 
    -\frac{\ell(p-1) +p(n-1)}{(p-1)(\ell+n)}.
    \]
    Multiplying both sides of \eqref{3.2} by 
    $\displaystyle{t^{p-1}F(\mu(t)) \mu(t)^{\theta_1}}$ 
    and then integrating over $(0, \tau)$, taking into account that  the function
    $\displaystyle   {  h(\ell)= F(\ell)\ell^{\theta_{1}} }$  is non decreasing, we get
    \begin{equation}
    \label{b_G}
    	\begin{split}
        	& \gamma_{n,\ell,p}	\int_0^\tau 
        	t^{p-1}F(\mu(t)) \, dt \\
        	& 
        	\leq
        	\int_0^\tau \bigl(- \mu'(t) \bigr) t^{p-1}\mu(t)^{\theta_1 } F(\mu(t)) \left(\int_0^{\mu (t)}f^\ast (s ) \, ds\right)^{\frac{1}{p-1}} \! dt \\
        		&
        		+F(\abs{\Omega}_{\ell})\frac{\abs{\Omega}_{\ell}^{\theta_1}}{p\tilde{\beta} ^{\frac{p}{p-1}}} \left(\int_0^{\abs{\Omega}_{\ell}}f^\ast (s ) \, ds
        	\right)^{\frac{p}{p-1}}.
    	\end{split} 
    \end{equation}
Let     
    \[
        C_2
        =
        F(\abs{\Omega}_{\ell})
        \frac{\abs{\Omega}_{\ell}^{\theta_1 }}{p\tilde{\beta} ^{\frac{p}{p-1}}} \left(\int_0^{\abs{\Omega}_{\ell}}f^\ast (s ) \, ds\right)^{\frac{p}{p-1}}.
    \]
   Integrating by parts both sides of \eqref{b_G} we derive
    \begin{equation}
        \label{Solidus_Snake}
        \begin{split}
            \gamma_{n,\ell,p} \,  \tau  \int_0^{\tau}  t^{p-2} F (\mu(t)) \, dt + \tau H_\mu(\tau) & \leq \gamma_{n,\ell,p} \int_{0}^{\tau} \int_0^t r^{p-2} F(\mu(r)) \, dr dt \\
            & + \int_0^{\tau} H_\mu(t) \, dt +  C_2 ,
        \end{split}
    \end{equation}
    where
    \[
    H_\mu(\tau)
    =
    -\int_{\tau}^{+\infty} t^{p-2} \mu(t)^{\theta_1} F(\mu(t)) \biggl( \int_0^{\mu(t)} f^*(s) \, ds \biggr)^{\frac{1}{p-1}} \, d\mu(t).
    \]
    Setting
    \begin{equation*}
    	\begin{multlined}
    \xi_2(\tau)=	
    	\gamma_{n,\ell,p}	\int_0^\tau \int_0^t  r^{p-2}F(\mu(r)) \, dr 
    	+\int_0^t H_{\mu}(t) \, dt,
    	\end{multlined} 
    \end{equation*}
    then \eqref{Solidus_Snake} reads as
    \[
    \tau  \xi_2'(\tau ) 
    \leq  \xi_2(\tau) +C_2.
    \]
    Lemma \ref{lemma_Gronwall}, with $\tau_0=v_m$, ensures that 
    the following inequality holds true for any $\tau \geq v_m $
    \begin{align*}
        & \gamma_{n,\ell,p}  \int_{0}^{\tau}  t^{p-2} F(\mu(t)) \, dt + H_\mu(\tau) \\
        & \leq \frac{\displaystyle{(p-1)\int_{0}^{v_m} t^{p-2} F(\mu(t)) \, dt +H_\mu(v_m) + C_2}}{v_m}
        \left( \frac{\tau}{v_m}\right)^{p-2}.
    \end{align*}
    The previous inequality becomes an equality if we replace $\mu(t)$ with $\phi(t)$, so, recalling that $\mu(t) \leq \phi(t)= \abs{\Omega}_{\ell}$  for $t \leq v_m$, we obtain
    \[
      \gamma_{n,\ell,p}  \int_{0}^{\tau}  t^{p-2} F(\mu(t)) \, dt + H_\mu(\tau) \leq 
         \gamma_{n,\ell,p} \int_{0}^{\tau} t^{p-2} F(\phi(t)) \, dt +H_\phi(\tau).
    \]
    Letting $\tau \to \infty$, we have
    \[
    \int_0^\infty t^{p-2} F(\mu(t)) \, dt \leq \int_0^\infty t^{p-2} F(\phi(t)) \, dt,
    \]
    since, as we will show,
    \begin{equation}
        \label{hto0}
        \lim_{\tau \to 0}H_\mu(\tau)=\lim_{\tau \to 0} H_\phi(\tau) = 0.
    \end{equation}
    This proves \eqref{scorciatoia}, and hence \eqref{L^2}.
    
    To prove \eqref{hto0}, recalling that $p \geq n$, we observe that
    	\begin{equation*}
    	    t^{p-2} \mu(t)=\int_{\Set{u>t}} t^{p-2} \abs{x}^\ell \, dx \leq \int_{\Set{u>t}} u^{p-2} \abs{x}^\ell \, dx  \leq \norma{u}_{L^p(\Omega,\abs{x}^\ell \, dx)}^{p-2} \, \mu(t)^{\frac{2}{p}},
    	\end{equation*}
    	therefore
    	\begin{align*}
    	    \abs{H_\mu(\tau)}
    	    &=
    	    \int_{\tau}^{+\infty} t^{p-2} F(\mu(t))
    	    \mu(t)^{\theta_1} \biggl( \int_0^{\mu(t)} f^*(s) \, ds \biggr) (-\mu'(t)) \, dt \\
    	    &\leq \biggl(\int_0^{\abs{\Omega}_{\ell}} \! f^*(s) \, ds\biggr) \norma{u}_{L^p(\Omega, \abs{x}^\ell dx)}^{p-2} \int_{\tau}^{+\infty} \! \! \! F(\mu(t))  \mu(t)^{\frac{2}{p}+\theta_1 -1} (-\mu'(t)) \, dt. 
    	\end{align*}
    	The claim follows by observing that the right-hand side of the above inequality goes to $0$ as $\tau \to +\infty$.
    	
\end{proof}


\begin{oss}
    We emphasize that, with the same argument in \cite{ANT, AGM}, it is possible to obtain a stronger comparison in Theorem \ref{teorema1}. It can be shown 
    \begin{equation*}
        \norma{u}_{L^{k,1}(\Omega,\abs{x}^\ell \, dx)} \, \leq \norma{v}_{L^{k,1}(\Omega^\sharp,\abs{x}^\ell \, dx)} \qquad \forall \, 0 < k \leq \frac{(\ell+n)(p-1)}{\ell(p-1)+p(n-1)}
    \end{equation*}
and    
    \begin{equation*}
        \norma{u}_{L^{pk,p}(\Omega,\abs{x}^\ell \, dx)} \, \leq \norma{v}_{L^{pk,p}(\Omega^\sharp,\abs{x}^\ell \, dx)} \qquad \forall \, 0 < k \leq \frac{(\ell+n)(p-1)}{\ell(p-1)+p(n-2) +n},
    \end{equation*}
where $\norma{\cdot}_{L^{p,q}(\Omega,\abs{x}^\ell \, dx)}$ is defined in \eqref{deff_norma_lorentz}.
\end{oss}

\vspace{0.5 cm}

\begin{proof}[Proof of Theorem \ref{teorema2}]{\color{white} .}
    Let
    \[
    \theta_2 = \frac{\ell(p-1) +p(n-1)}{p(\ell+n)}
    \]
    and 
    \[
    \delta_2 = -\left(\theta_2-\frac{1}{p}\right)\frac{p}{p-1}.
    \]
    We point out that $\delta_2 \geq 0$ by assumption \eqref{l<n+rob}.
    
    Obviously in this case we have  
    \[
    \displaystyle{\int_0^{\mu(t)} f^\ast(s) \, ds= \mu(t)}.
    \]
    Hence \eqref{3.2} can be written as
    \begin{equation}
    	\label{f=1}
    	\gamma_{n,\ell,p} \,
    	\mu(t)^{\left(\theta_2-\frac{1}{p}\right)\frac{p}{p-1}} \leq - \mu'(t) + \frac{1}{\tilde{\beta} ^{\frac{1}{p-1}}}\int_{\partial U_t^\text{ext}} \frac{\abs{x}^{ \frac{\ell}{p'} }  }{u} \, d\mathcal{H}^{n-1}(x).
    \end{equation}
    Multiplying both sides of \eqref{f=1} by $t^{p-1} \mu(t)^{\delta_2}$ and, then, integrating  from $0$ to $\tau \geq v_m$, we obtain
    \begin{equation}
    	\label{20}
    	\begin{split}
        	\gamma_{n,\ell,p} \int_0^\tau  t^{p-1} \, dt &
    	    \leq
    	    \int_0^\tau t^{p-1} 
    	    \mu(t)^{\delta_2}
    	    (- \mu'(t)) \, dt \\
    	    & + 
    	    \frac{1}{\tilde{\beta} ^{ \frac{1}{p-1}}}\int_0^\tau t^{p-1} 
    	    \mu(t)^{\delta_2}
    	    \int_{\partial U_t^\text{ext}} \frac{1}{u} \, d\mathcal{H}^{n-1}(x) \\
    		&
    		\leq
    		\int_0^\tau t^{p-1}
    		\mu(t)^{\delta_2} (- \mu'(t)) \, dt 
    		+  \frac{\abs{\Omega}_{\ell}^{1+ \delta_2}}{p \tilde{\beta} ^{\frac{p}{p-1}}}.
    	\end{split}
    \end{equation}
		
    Taking into account of Lemma \ref{lemma3.3}, if we replace $\mu(t)$ with $\phi(t)$, the previous inequality holds as equality.
    Hence, we get
    \begin{equation*}
        \int_0^\tau t^{p-1} 
        \mu(t)^{\delta_2} (- \mu'(t)) \, dt  
        \geq
        \int_0^\tau t^{p-1} \phi(t)^{\delta_2} (- \phi'(t)) \, dt .
    \end{equation*}
    In turn an integration by parts yields 
    \begin{equation*}
        \begin{split}
            & -\tau^{p-1} \, \frac{\mu(\tau)^{1+ \delta_2}}{1+ \delta_2} + (p-1) \int_0^\tau t^{p-2}\, \frac{\mu(t)^{1+ \delta_2}}{1+ \delta_2} \, dt \geq \\
            & -\tau^{p-1} \, \frac{\phi(\tau)^{1+ \delta_2}}{1+ \delta_2} + (p-1) \int_0^\tau t^{p-2} \, \frac{\phi(t)^{1+ \delta_2}}{1+ \delta_2}\, dt.
        \end{split}
    \end{equation*}
	The above inequality  allows us to use claim (ii) of the Gronwall's Lemma, this time to the function
	\[
	    \xi_3(\tau)
	    =
	    \int_0^\tau s^{p-2}
	    \left(
	    \frac{\mu(s)^{1+ \delta_2}-\phi(s)^{1+ \delta_2}}{1+ \delta_2}
	    \right)\, ds.
	\]
	Hence for any $\tau \geq v_m$ it holds
	\[
    	\tau^{p-2}
    	\left(
    	\frac{\mu^{1+ \delta_2}(\tau)- \phi^{1+ \delta_2}(\tau)}{1+ \delta_2}\right)
    	\leq
    	(p-1) \frac{\tau^{p-2}}{v_m^{p-2}}\int_0^{v_m} \! \! \! s^{p-2} 
    	\left(
    	\frac{\mu^{1+ \delta_2}(s)- \phi^{1+ \delta_2}(s)}{1+ \delta_2}
    	\right)  ds.
	\]
	Inequality \eqref{minima_eq} ensures us that the right-hand side of the inequality above is non-positive, therefore
	\begin{equation*}
	    \mu (\tau) \leq \phi (\tau) \quad \forall \tau \geq v_m.
	\end{equation*} 
	Finally, since by \eqref{minima_eq} we have
	\[
	\mu (\tau) \leq \phi (\tau) = \abs{\Omega}_{\ell} \quad \forall \tau \leq v_m,
	\]
	and, therefore, the inequality above holds true 
	in $\left[ 0,+\infty \right) $. Claim \eqref{7} is hence proven.

\end{proof}

%

\section{Some applications}
\label{section_4}
In this section we give some applications of our previous results. First of all, as already mentioned in the introduction, it is possible to derive a  Faber-Krahn inequality for the Robin $p$-Laplacian operator with non-constant boundary parameter.







	
\begin{proof}[Proof of Theorem \ref{Faber_Krahn}]
	Let $u$ be a positive minimizer  of \eqref{trace}, then 
	\begin{equation*}
    	\begin{cases}
        	-\Delta_p u= \lambda_{1,\beta}(\Omega) \, u^{p-1}  \abs{x}^{\ell} & \text{ in } \Omega \\[0.3em]
        	\abs{\nabla u}^{p-2} \displaystyle{\frac{\partial u}{\partial \nu}} + \beta(x) 
        	u^{p-1}  =0  & \text{ on } \partial \Omega.
    	\end{cases}
	\end{equation*}
	Now, let $z$ be a solution to the following problem 
	\begin{equation}
    	\label{bho}
    	\begin{cases}
        	-\Delta_p z= \lambda_{1,\beta}(\Omega) \,  ( u^\sharp )^{p-1} \,  \abs{x}^{\ell}  & \text{ in } \Omega^\sharp \\[0.3em]
        	\abs{\nabla z}^{p-2} \displaystyle{\frac{\partial z}{\partial \nu}} + \tilde{\beta}
        	( r^\sharp )^{\frac{\ell}{p'}}  z^{p-1} =0  & \text{ on } \partial \Omega^\sharp.
    	\end{cases}
	\end{equation} 
	Therefore, Theorem \ref{teorema1} ensures that
    \[
	    \int_{\Omega} u^p \abs{x}^{\ell}\, dx = \int_{\Omega^\sharp} (u^\sharp)^p \abs{x}^{\ell}\, dx \leq \int_{\Omega^\sharp} z^p \abs{x}^{\ell}\, dx,
	\]
	while H\"older inequality gives
	\begin{align*}
	    \int_{\Omega^\sharp}(u^\sharp)^{p-1}  z\abs{x}^{\ell}  \, dx & \leq \left(\int_{\Omega^\sharp} ( u^\sharp )^p\abs{x}^{\ell} \, dx \right)^{\frac{p-1}{p}}  \left(\int_{\Omega^\sharp} z^p \abs{x}^{\ell}  \, dx\right)^{\frac{1}{p}} \\
	    & \leq \int_{\Omega^\sharp} z^p \abs{x}^{\ell} \, dx. 
	\end{align*}
	Hence, writing $\lambda_{1,\beta(x)}(\Omega)$ according to \eqref{bho}, we get
	\begin{equation*}
    	\begin{split}
        	\lambda_{1,\beta(x)}(\Omega)&= \frac{\displaystyle{\int_{\Omega^\sharp} \abs{\nabla z}^p \, dx +  \int_{\partial \Omega^\sharp}\tilde{\beta}  (r^\sharp)^{\frac{\ell}{p'}} z^p\  \, d\mathcal{H}^{n-1}(x)}}{\displaystyle{\int_{\Omega^\sharp} (u^\sharp)^{p-1}   z \abs{x}^{\ell} \, dx}} \\
        	&\geq 
        	\frac{\displaystyle{\int_{\Omega^\sharp} \abs{\nabla z}^p \, dx +  \int_{\partial \Omega^\sharp} \tilde{\beta} (r^\sharp)^{\frac{\ell}{p'}}z^p \, d\mathcal{H}^{n-1}(x)}}{\displaystyle{\int_{\Omega^\sharp} z^p\abs{x}^{\ell}  \, dx}} \geq\lambda_{1,\tilde{\beta}}(\Omega^\sharp).
    	\end{split}
	\end{equation*}
\end{proof}

\section{Appendix}
\label{appendix:A}
Here we provide the proofs of Propositions \ref{proposizione_2} and \ref{proposizione_3} just for
$p=n$, since the remaining cases can be treated analogously.

\begin{proof}[Proof of Proposition \ref{proposizione_2}]
     We recall that, as well-known,  $W^{1,n}(\Omega)$ is compactly embedded in $L^{q}(\Omega)$ for every $q$ finite.
    Then  H\"older inequality gives
    \begin{equation*}
        \left(\int_\Omega \abs{\psi}^n \abs{x}^\ell\right)^{\frac{1}{n}} \leq \left(\int_\Omega \abs{\psi}^q\right)^{\frac{1}{q}}\left(\int_\Omega \abs{x}^{\frac{\ell}{n}s}\right)^{\frac{1}{s}}\quad \text{ where }\quad \frac{1}{n}= \frac{1}{q }+\frac{1}{s}.
    \end{equation*}
    If $q \geq 0$, then $s \geq n$ and so
    \[
    q > \frac{n^2}{n-\abs{\ell}} \quad \implies \quad \frac{ \abs{\ell} }{n} s < n.
    \]
    Hence
    \[
    \norma{\psi}_{L^n(\Omega, \abs{x}^\ell \, dx)} \leq D_1 \norma{\psi}_{L^q(\Omega)} \leq C_1\norma{\psi}_{W^{1,n}(\Omega)},
    \]
    and, therefore, the embedding is at least continuous.
    
    Now, if we take a sequence $\{\psi_k\}_k$ bounded in $W^{1,n}(\Omega)$, by the  compact embedding of  $W^{1,n}(\Omega)$ in $L^{q}(\Omega)$, up to a subsequence, $\{\psi_{k}\}_k$  strongly converges in $L^q(\Omega)$.
    The continuous embedding of $L^q(\Omega)$  in $L^n(\Omega, \abs{x}^\ell \, dx)$ guarantees us the strong convergence in the second space as well.
\end{proof}

\begin{proof}[Proof of Proposition \ref{proposizione_3}]
    The solution to \eqref{p_originale} is the unique minimum of the functional

\begin{equation}
    \label{deff_funzionale_G}
    \begin{split}
        G : \psi \in W^{1,n}(\Omega) & \rightarrow
        \frac{1}{n} \int_{\Omega} \abs{\nabla \psi}^n \, dx + 
        \frac{1}{n} \int_{\partial \Omega} \beta(x) \abs{ \psi}^n  \, d\mathcal{H}^{n-1}(x) \\
        &- \int_{\Omega} f \psi \abs{x}^\ell \, dx.
    \end{split}
\end{equation}
Existence can be achieved by means of direct methods of calculus of variation, see for instance \cite{Dac, G, Lind,AGM}, while one can prove uniqueness by using the same arguments contained in \cite{Belloni_Kawohl, Diaz_Saa}.

Here we just recall the proof of the existence.

    First of all we notice that functional $G$ is bounded from below. Indeed by \eqref{lower_bon} and since $0 \not\in \partial\Omega$, we have
    \begin{align*}
        G(\psi) & = \frac{1}{n} \int_{\Omega} \abs{\nabla \psi}^n \, dx 
        + \frac{1}{n} \int_{\partial \Omega} \beta(x) 
        \abs{ \psi}^n \, d \mathcal{H}^{n-1}(x) - \int_{\Omega} f \psi \abs{x}^\ell \, dx \\
        & \geq \frac{1}{n} \int_{\Omega} \abs{\nabla \psi}^n \, dx 
        + C \int_{\partial \Omega} \abs{ \psi}^n - \int_{\Omega} f \psi \abs{x}^\ell \, dx \\
        & \geq C \norma{\psi}_{W^{1,n}(\Omega)}^n - \int_{\Omega} f \psi \abs{x}^\ell \, dx ,
    \end{align*}
    where in the last inequality we have used the well-known trace embedding Theorem.
    Using H\"older inequality and Proposition \ref{proposizione_2}, we have
    \begin{equation*}
        \label{stima_G_psi}
        G(\psi) \geq C \norma{\psi}_{W^{1,n}(\Omega)}^n - \norma{\psi}_{W^{1,n}(\Omega)} \norma{f}_{L^{n'}(\Omega, \abs{x}^\ell \, dx)}
    \end{equation*}
    By Young parametric inequality we finally get
    \begin{align}
    \label{g>m}
        G(\psi) 
        \geq \left( C - \frac{\varepsilon^n}{n} \right) \norma{\psi}_{W^{1,n}(\Omega)}^n - \frac{C}{\varepsilon^{n'} n'} \norma{f}_{L^{n'}(\Omega, \abs{x}^\ell \, dx)}
    \end{align}
    For  $\varepsilon$ small enough we have
    \[
    G(\psi) \geq - \frac{1}{\varepsilon n'} \, \norma{f}_{L^{n'}(\Omega, \abs{x}^\ell \, dx)} > - \infty .
    \]
    \noindent Now let $\Set{u_k}_k$ a minimizing sequence for $G$, we can assume, without loss of generality, that $G(u_k) \leq m_G + 1$ where
    \begin{equation*}
    m_{G}=\inf_{u\in W^{1,n}(\Omega )}G(u).
    \end{equation*}
     So, by \eqref{g>m} for a fixed $\varepsilon$ small, we have
    \[
    m_{G} +1+ \frac{C}{\varepsilon^{n'} n'} \norma{f}_{L^{n'}(\Omega, \abs{x}^\ell \, dx)} \geq \left( C - \frac{\varepsilon^n}{n} \right) \norma{u_k}_{W^{1,n}(\Omega)}^n.
    \]
    So the sequence $\Set{u_k}_k$ is bounded in $W^{1,n}(\Omega)$ hence, up to a subsequence, it converges weakly in $W^{1,n}(\Omega)$ and strongly in $L^n(\Omega,\abs{x}^\ell\, dx)$ and in $L^n(\partial \Omega)$ to a function $u$.
    
    Thus, we have
    \begin{align*}
        \liminf_k \int_{\Omega} \abs{\nabla u_k}^n \, dx &\geq \int_{\Omega} \abs{\nabla u}^n \, dx \\
        \liminf_k \int_{\Omega} \beta(x) u_k^n \, d\mathcal{H}^{n-1}(x) & = \int_{\partial \Omega} \beta(x) u^n \, d\mathcal{H}^{n-1}(x) \\
        \liminf_k \int_{\Omega} u_k f \abs{x}^\ell \, dx &= \int_{\Omega} u f \abs{x}^\ell \, dx
    \end{align*}
    and finally
    \[
    m_{G} = \liminf_k G(u_k) \geq G(u) = m_{G}.\qedhere
    \]
\end{proof}

\noindent \textbf{Founding} This work has been partially supported by the PRIN project 2017JPCAPN (Italy) Grant: “Qualitative and quantitative aspects of nonlinear PDEs” and by GNAMPA of INdAM.

\vspace{-1em}

\addcontentsline{toc}{chapter}{Bibliografia}

\printbibliography[heading=bibintoc, title={References}]

\renewcommand{\abstractname}{}    

\begin{abstract}
    \textsc{Dipartimento di Matematica e Applicazioni “R. Caccioppoli”, Universita` degli Studi di Napoli “Federico II”, Complesso Universitario Monte S. Angelo, via Cintia - 80126 Napoli, Italy.}
    
    \textsf{e-mail: vincenzo.amato@unina.it}
	
	\textsf{e-mail: francesco.chiacchio@unina.it}
	
	\vspace{0.2cm}
	
	\textsc{Mathematical and Physical Sciences for Advanced Materials and Technologies, Scuola Superiore Meridionale, Largo San Marcellino 10, 80126 Napoli, Italy.}
	
	\textsf{e-mail: andrea.gentile2@unina.it}
\end{abstract}

\pagebreak 

\let\clearpage\relax

\end{document}